\documentclass[11pt]{amsart}

\usepackage[T1]{fontenc}
\usepackage[utf8]{inputenc}
\usepackage{amsmath,amssymb,amsfonts,amsthm,mathtools}
\usepackage[protrusion=true,expansion=false]{microtype}
\usepackage[colorlinks=true,citecolor=blue,linkcolor=blue,urlcolor=blue]{hyperref}

\title[Maximum occupation time of transient ERW on $\mathbb Z$]{Maximum occupation time of a transient excited random walk on $\mathbb Z$}

\author{Reza Rastegar}
\address{Woodinville, WA, USA}
\email{reza.j.rastegar@gmail.com}

\author{Alexander Roitershtein}
\address{Department of Mathematics, Texas A\&M University, College Station, TX 77843, USA}
\email{alexander@stat.tamu.edu}

\date{May 25, 2026}

\theoremstyle{plain}
\newtheorem{theorem}{Theorem}[section]
\newtheorem{lemma}[theorem]{Lemma}

\theoremstyle{definition}
\newtheorem{assumption}[theorem]{Assumption}

\theoremstyle{remark}
\newtheorem{remark}[theorem]{Remark}

\newcommand{\one}[1]{\mbox {\bf 1}_{\{#1\}}}
\newcommand{\pp}{{\mathbb P}}

\newcommand{\witi}{\widetilde}

\newcommand{\zz}{{\mathbb Z}}
\newcommand{\nn}{{\mathbb N}}
\newcommand{\ee}{{\mathbb E}}
\newcommand{\rr}{{\mathbb R}}
\newcommand{\calf}{{\mathcal F}}
\newcommand{\veps}{\varepsilon}
\newcommand{\beq}{\begin{eqnarray*}}
	\newcommand{\feq}{\end{eqnarray*}}
\newcommand{\beqn}{\begin{eqnarray}}
	\newcommand{\feqn}{\end{eqnarray}}
\newcommand{\as}{\mbox{\rm a.\,s.}}
\newcommand{\io}{\quad \mbox{\rm i.\,o.}}
\newcommand{\speed}{{\mbox{\rm v}}}

\begin{document}
	\begin{abstract}
		We consider a transient excited random walk on $\mathbb Z$ and study the asymptotic behavior of the occupation time of a currently most visited site. In particular, our results imply that, in contrast to the random walks in random environment, a transient excited random walk does not spend an asymptotically positive fraction of time at its favorite (most visited up to a date) sites.
	\end{abstract}
	
	\maketitle
	
	\section{Introduction and statement of results}
	\label{intro}
	Excited random walks, also called random walks in cookie environments, are modifications of nearest-neighbor simple random walk in which, during the first several visits to each site of the integer lattice, the jump kernel may prefer one direction and assigns
	equal probabilities to the remaining $(2d-1)$ directions. If the current location of the random walk has been already visited more than a certain number of times, then the walk moves to one of its nearest neighbors
	with equal probabilities. The model was introduced by Benjamini and Wilson in \cite{bwilson} and extended by Zerner in \cite{martin1}.
	\par
	In this paper we focus on the excited random walks in dimension one.
	To define the transition mechanism of the random walk, fix an integer $M\in \nn$ and let
	\beq
	\Omega_{M}
	&=& \Bigl\{ \omega=(\omega(z,i))_{z\in\zz,\,i\in\nn}:\\
	&& \omega(z,i)\in[0,1] \text{ for } z\in\zz,\ 1\leq i\leq M,\\
	&& \text{and }\omega(z,i)=1/2 \text{ for } z\in\zz,\ i>M\Bigr\}.
	\feq
	The value of $\omega(z,i)$ determines the probability of the jump from $z$ to $z+1$ upon the $i$-th visit of the random walk to the site $z\in\zz.$
	The random walk is assumed to be nearest neighbor, and hence the probability of the jump from $z$ to $z-1$ upon its $i$-th visit to $z$
	is given by the complementary probability $1-\omega(z,i).$ The elements of the set $\Omega_M$ are called {\em cookie environments}.
	\par
	For a fixed $z\in\zz,$ $\bigl(\omega(z,i)\bigr)_{1\leq i\leq M}$ can be thought of as a sequence of numerical characteristics called ``strengths",
	associated with a pile of $M$ ``cookies" placed at $z.$ Correspondingly, $\omega(z,i)$ is
	referred to as the {\em strength of the $i$-th cookie} at the pile. With a slight abuse of language, using the above introduced jargon we will often
	identify the strength $\omega(z,i)$ of a cookie with the cookie itself.
	The transition kernel of the random walk can be informally described as follows:
	while the supply of the cookies at a given site lasts, the walker eats a cookie upon each visit there and then makes one step
	in a random direction, such that the probability of moving to the right is equal to the ``strength" of the cookie just eaten.
	\par
	More precisely, the random walk in a cookie environment $\omega\in \Omega_M$
	is defined as follows. Denote $\nn_0=\nn\cup\{0\},$ let $\Sigma=\zz^{\nn_0}$ be the state space of the infinite paths of a discrete-time random walk on $\zz$
	($\zz$ for the location and $\nn_0$ for the time), and let $\calf$ be its Borel $\sigma$-algebra (i.~e., the $\sigma$-algebra generated by the cylinder sets of the infinite product space $\Sigma$).
	For any $x\in \zz$ and $\omega\in \Omega_M,$ an excited random walk (abbreviated in what follows as ERW) starting at $x\in \zz$ in the cookie environment $\omega$
	is a sequence of random variables $X=\bigl(X_n\bigr)_{n\in \nn_0}$ defined on $(\Sigma, \calf, P_{x,\omega})$ such that $P_{x,\omega}(X_0=x)=1$ and
	\beq
	P_{x,\omega} (X_{n+1}=X_n+1 | \calf_n ) = \omega(X_n, \ell_n),
	\feq
	where $\calf_n:= \sigma(X_i,\, 0\leq i\leq n)$ and
	\beqn
	\label{ell-current}
	\ell_n:=\#\bigl\{0\leq i\leq n: X_i = X_n\bigr\}
	\feqn
	is the number of visits to the current site by time $n.$
	The measure $P_{x,\omega}$ is usually referred to as the {\em quenched law} of the excited random walk in the
	cookie environment $\omega.$ Let $\pp$ be a probability measure on $\Omega_M$
	that makes the collection of ``piles" $\omega_z:=(\omega(z,i))_{i\in\nn}$ indexed by $z\in\zz$ into an i.i.d. sequence.
	Notice that we do not insist on the independence of the cookies within a given pile, that is
	the random variables $\omega(z,i)$ for a fixed $z\in\zz$ can be dependent under $\pp.$
	The (associated with $\pp$) {\em annealed} (average) law $P_x$ of the ERW on $(\Sigma,\calf)$ is defined by setting
	$P_x(\,\cdot\,)=\ee[P_{x,\omega}(\,\cdot\,)],$ where $\ee$ is the expectation induced by the probability law $\pp.$
	\par
	Many important aspects of the asymptotic behavior of excited random walks on $\zz$ are by now well-understood.
	In particular, Zerner in \cite{martin1}, Basdevant and Singh in \cite{bsingh1,bsingh2}, Kosygina and Zerner in \cite{kosygina1}, and Kosygina and Mountford in
	\cite{kosygina2} characterized the recurrence-transience behavior and possible speed regimes of the ERW, and proved limit theorems for the
	fluctuations of the current location of the ERW.  The goal of this paper is to study the asymptotic dynamics of the occupation
	time of a currently most visited site for a transient ERW.
	\par
	The consumption of a cookie $\omega(z,i)$ creates the local drift (i.e., the bias in the conditional on the history $\calf_n$ expectation
	of the subsequent displacement) equal to $E_{x,\omega}\bigl[X_{n+1}-X_n\bigm|X_n=z,\ell_n=i\bigr]=2\omega(z,i)-1.$
	Let
	\beq
	\delta= \ee \Bigl[ \sum_{i=1}^M \bigl(2\omega(z,i)-1\bigr) \Bigr]
	\feq
	be the annealed expectation of the total drift (i.e., the ``boost" in the positive direction) available to the random walk at a site $z\in\zz.$
	Since the piles $\omega_z$ are assumed i.i.d., $\delta$ is independent of $z.$
	Following \cite{kosygina1} and \cite{kosygina2} (and in contrast to the original version presented in \cite{bwilson,martin1}) we do not impose
	the condition that $\pp\bigl(\omega(z,i)\geq 1/2\bigr)=1,$ thereby allowing both ``positive" and ``negative" cookies.
	It turns out (see \cite{kosygina2,kosygina1}) that the asymptotic behavior of a one-dimensional excited random walk is largely determined by the
	value of the parameter $\delta.$ In particular, under a mild non-degeneracy assumption on the cookie environment, the random walk is transient
	to the right if and only if $\delta>1.$
	\par
	Throughout this paper we will impose the following conditions on the cookie environment.
	\begin{assumption}
		\label{ass1}
		The following assumptions hold:
		\begin{itemize}
			\item[(a)] Independence: the sequence of ``piles" $\omega_z=(\omega(z,i))_{i \in \nn}$ indexed by sites $z\in\zz$ is an i.i.d. sequence under $\pp$.
			\item[(b)] Non-degeneracy: $\ee\bigl[\prod_{i=1}^M\omega(0,i) \bigr]>0\,$ and $\,\ee\bigl[\prod_{i=1}^M(1-\omega(0,i)) \bigr]>0$.
			\item[(c)] Transience: $\delta>1.$
		\end{itemize}
	\end{assumption}
	The non-degeneracy condition rules out the two one-sided degenerate cases at the level of cookie piles and is the standard ellipticity-type assumption used in the one-dimensional ERW transience theory. It is known \cite{kosygina1,martin1} that under Assumption~\ref{ass1} the ERW is transient to the right
	(that is, $P_0\bigl(\lim_{n\to\infty} X_n=\infty\bigr)=1$) and, furthermore, has the asymptotic speed
	\beqn
	\label{rwspeed}
	\speed:= \lim_{n\rightarrow \infty}\frac{X_n}{n}\in [0,1),
	\feqn
	which is strictly positive if and only if $\delta>2.$
	\par
	Define the occupation time of the ERW at site $x\in\zz$ as
	\beqn
	\label{localtime}
	\xi_n(x):= \# \bigl\{ 0\leq i \leq n\,:\, X_i=x  \bigr\}.
	\feqn
	Thus $\xi_n(x)$ is the number of times that the ERW visits $x\in\zz$ during the first $n$ steps. Let
	\beqn
	\label{suplocaltime}
	\xi^*_n:=\max_{x\in\zz}{\xi_n(x)}
	\feqn
	be the largest number of visits to a single site during the first $n$ steps. For the sake of notational convenience
	we will occasionally write $\xi^*(n)$ for $\xi_n^*.$
	The asymptotic properties of the process $\xi^*:=(\xi^*_n)_{n\in \nn}$ can be compared to
	those of the simple random walk as well as of a random walk in random environment (abbreviated in what follows as RWRE) in dimension one.
	For a comprehensive up-to-date review of the former topic, see a monograph of R\'{e}v\'{e}sz \cite{revesz1}.
	For more recent developments on RWRE, we refer to \cite{nina2, nina1, nina3} and references therein.
	\par
	This paper is devoted to the study of the limit points of the sequence $\xi^*$ for a transient ERW. The approach adopted here is based on a reduction of the
	study of the asymptotic behavior of the ERW to that of a branching process and the subsequent reformulation of the problem
	in terms of the asymptotic dynamics of the most populated generation of the branching process.
	Similar questions for transient RWRE have been addressed in the interesting paper by Gantert and Shi \cite{nina1}.
	The essential branching processes machinery which enables the implementation of our approach to excited random walks was introduced in \cite{bsingh1,bsingh2}
	and further developed in \cite{kosygina2,kosygina1}.
	\par
	Our first result concerns non-ballistic ERW.
	\begin{theorem}
		\label{thm2}
		Suppose that Assumption~\ref{ass1} holds with $\delta\in (1,2].$ Then,
		\begin{itemize}
			\item [(i)] The following holds:
			\beqn
			\label{positive}
			\limsup_{n\to\infty}\frac{\xi^*_n}{n^{1/2}} > 0, \qquad P_0-\as
			\feqn
			\item [(ii)] Furthermore, for any $\alpha>\frac{1}{\delta}:$
			\beq
			\lim_{n\to\infty}\frac{\xi_n^*}{n^{1/2}(\log  n)^\alpha}=0
			\qquad
			\mbox{\rm while}
			\qquad
			\lim_{n\to\infty}\frac{(\log {n})^{\alpha} \xi_n^*}{n^{1/2}}=\infty, \qquad P_0-\as
			\feq
		\end{itemize}
	\end{theorem}
	The above theorem implies in particular that unlike RWRE (see \cite{nina1}), a non-ballistic transient ERW  does not spend a positive fraction of time at its favorite sites.
	While the asymptotic behavior of $\xi_n^*$ for transient RWRE seems to be determined by the so-called ``traps" created by a random potential (cf.
	\cite{nina1} and \cite{nina2}), and is radically different from that of the simple unbiased random walk,
	the limsup asymptotic of $\xi_n^*$ for a non-ballistic transient ERW turns out to be rather similar to its counterpart
	for a simple unbiased random walk (cf. Theorem~11.3 in \cite{revesz1}). We also remark that, based on a comparison with the latter, we believe that in fact $\limsup_{n\to\infty}\frac{\xi^*(n)}{n^{1/2}}=\infty$ under the conditions of Theorem~\ref{thm2},
	but we are unable to prove this conjecture.
	\par
	The next theorem deals with the asymptotic behavior of $\xi^*(n)$ for ballistic ERW.
	Recall that, for $\rho \in\rr,$ a sequence $\phi:\nn\to\rr$ is {\em regularly varying (at infinity) of index $\rho$}
	if $\phi(n)=n^\rho L(n)$ for some sequence $L:\nn\to\rr$ such that $\lim_{n\to\infty} L(\lfloor\lambda n\rfloor)/L(n)=1$
	for all $\lambda >0$ (that is, $L$ is {\em slowly varying}).
	Here and henceforth, $\lfloor x \rfloor$ denotes the integer part of a real number $x,$ that is $\lfloor x \rfloor=\max\{n\in\zz:n\leq x\}.$
	\begin{theorem}
		\label{thm1}
		Suppose that Assumption~\ref{ass1} holds with $\delta>2.$
		\begin{itemize}
			\item [(i)] Let $\phi:\nn\to\rr$ be a regularly varying sequence of a positive index $\rho>0.$
			Then,
			\beq
			\limsup_{n\to\infty}\frac{\xi^*_n}{\phi(n)} =
			\left\{
			\begin{array}{l}
				0 \\
				\infty
			\end{array}
			\right.
			\quad P_0-\as \qquad \text{\rm if and only if} \qquad
			\sum_{n=1}^\infty{\frac{1}{[\phi(n)]^{\delta}}}\,
			\left\{
			\begin{array}{l}
				<\infty
				\\
				= \infty. \\
			\end{array}
			\right.
			\feq
			In particular, for any $\alpha>1/\delta,$
			\beq
			\lim_{n\to\infty} \frac{\xi_n^*}{n^{1/\delta}(\log  n)^\alpha} =0, \qquad P_0-\as,
			\feq
			and hence $\lim_{n\to\infty}\frac{\xi_n^*}{n}=0,$ $P_0-\as$
			\item [(ii)] For $\alpha>1/\delta,$
			\beq
			\lim_{n\to\infty} \frac{(\log  n)^\alpha\xi_n^*}{n^{1/\delta}} =\infty, \qquad P_0-\as
			\feq
		\end{itemize}
	\end{theorem}
	The rest of the paper is organized as follows. In Section~\ref{reduction}, we consider an auxiliary branching process
	formed by successive level crossings along the random walk path.
	The proofs of Theorems~\ref{thm2} and~\ref{thm1} are contained in Section~\ref{proofs}.
	\section{Reduction to a branching process}
	\label{reduction}
	The proof of our main results, given in this section, relies on the use of
	a mapping of the paths of ERW into realizations of a suitable branching process with migration.
	In this section we discuss the branching process framework and recall some auxiliary results related
	to it; see \cite{kosygina2,kosygina1} or expository surveys \cite{erwnotes, kosygina3} for more details.
	\par
	For $m\in\nn,$ let $T_m$ be the first hitting time of site $m,$ that is
	\beq
	T_m=\inf\bigl\{n\in\nn: X_n=m \bigr\}.
	\feq
	Since the ERW is transient to the right under Assumption~\ref{ass1}, the random variables $T_m$ are almost surely finite
	for all $m\in\nn$ under the law $P_0.$ Moreover, it follows from \eqref{rwspeed} (by passing to the random subsequence of indices
	$n=T_m$ in \eqref{rwspeed}) that
	\beqn
	\label{rwspeed1}
	\lim_{n\to\infty} \frac{T_n}{n}=\lim_{m\to\infty} \frac{T_m}{X_{T_m}}=\speed^{-1}\in[1,\infty],\qquad P_0-\as
	\feqn
	Set now $D_m^m:=0$ and for $k\leq m-1$ let
	\beq
	D_k^m= \sum_{i=0}^{T_m-1}\one{X_{i+1}=k-1,\, X_i=k}
	\feq
	be the number of down-crossing steps of the ERW from site $k$ to $k-1$ before time $T_m.$
	Then (see, for instance, \cite{nina1,kosygina1}),
	\beqn
	\label{tm_rep}
	T_m = m + 2\sum_{k\leq m}{D_k^m}=m + 2\sum_{0\leq k\leq m}{D_k^m}+2\sum_{k<0}{D_k^m}.
	\feqn
	It follows from \eqref{localtime} that
	\beq
	\xi_{T_m}(k) = \left\{
	\begin{array}{lll}
		0 &~\mbox{for}&k>m,\\
		1 &~\mbox{for}&k=m,\\
		D_k^m+D_{k+1}^m+\one{k\geq 0}&~\mbox{for}&k<m,
	\end{array} \right.
	\feq
	and hence
	\beqn
	\label{mineq}
	\max_{0\leq k\leq m} D_k^m &\leq& \xi^*_{T_m} = \max\Bigl\{1,\,\max_{k<m}{(D_k^m+D_{k+1}^m+\one{k\geq 0})}\Bigr\} \nonumber \\
	&\leq& 1+ 2\max_{0\leq k \leq m} D_k^m+2\max_{k<0} D_k^m.
	\feqn
	Notice that $\max_{k<0} D_k^m$ is bounded above by the total time spent by the random walk on the negative half-line.
	Since the random walk is transient to the right, the latter quantity is $P_0-\as$ finite.
	Therefore, for any eventually increasing non-negative sequence $\phi:\nn\to \rr_+,$ we have
	\beqn
	\label{fineq}
	\limsup_{m\to\infty}\frac{\max_{0\leq k\leq m} D_k^m}{\phi(m)} &\leq& \limsup_{m\to\infty}\frac{\xi^*_{T_m}}{\phi(m)}
	\leq 2 \limsup_{m\to\infty}\frac{\max_{0\leq k \leq m} D_k^m}{\phi(m)}.
	\feqn
	Similar inequalities hold with the $\limsup$ replaced by the $\liminf.$
	\par
	To elucidate the probabilistic structure of the sequence $D_k^m,$ it is convenient to exploit the following
	alternative definition of the random walk $(X_n)_{n\geq 0}.$ Assume that the underlying probability space is enlarged
	to include a sequence of, conditionally on $\omega,$ independent Bernoulli random variables (``coins") $\bigl(B(z,i)\bigr)_{i\in\nn,z\in\zz}$ such that
	\beqn
	\label{coin-tossing}
	P_{0,\omega}\bigl(B(z,i)=1\bigr)=\omega(z,i)\qquad \mbox{and}\qquad P_{0,\omega}\bigl(B(z,i)=-1\bigr)=1-\omega(z,i).
	\feqn
	Then the ERW $X$ can be alternatively defined by specifying the jump sequence recursively, as follows:
	\beqn
	\label{steps}
	X_{n+1}=X_n+B\bigl(X_n,\ell_n),
	\feqn
	where $\ell_n$ is introduced in \eqref{ell-current}. We adopt the terminology of \cite{kosygina2,kosygina1} and
	refer to the event $\{B(z,i)=1\}$ as a ``success" and to the event $\{B(z,i)=-1\}$ as a ``failure". For $z\geq 0,$ denote by
	$F^{(z)}_m$ the number of failures before the $m$-th success in the sequence $B^{(z)}:=\bigl(B(z,i)\bigr)_{i\in\nn}.$
	Let $V:=(V_k)_{k\geq 0}$ be a Markov chain on $\nn_0$ with transition kernel defined (under the law $P_0$) by
	means of the following recursion:
	\beq
	V_{k+1}=F_{V_k+1}^{(k)},\qquad k \geq 0.
	\feq
	The process $V$ can be thought of as a branching process with the following properties:
	\begin{enumerate}
		\item There is exactly one immigrant in each generation and the immigration happens before the reproduction.
		\item The number of offspring of $m$-th individual in generation $k\in \nn,$ which we denote by $W^{(k)}_m,$ is equal to $F^{(k)}_m-F^{(k)}_{m-1}.$ For future reference, we write
		\beqn
		\label{wmk}
		W^{(k)}_m=F^{(k)}_m-F^{(k)}_{m-1}.
		\feqn
	\end{enumerate}
	For non-negative reals $x\geq 0,$ denote by $P_x^V$ the law of the process $V$ that starts with $\lfloor x\rfloor$ individuals in the generation zero.
	It turns out that for every $n \geq 0,$ the distribution of $(V_0,V_1,...,V_n)$ under $P_0^V$ coincides with the distribution of the array $(D_n^n,D_{n-1}^n,...,D_0^n)$ associated with a transient to the right ERW (see, for instance, Section~2 in \cite{kosygina2}):
	\beqn
	\label{hren}
	(V_0,V_1,...,V_n) \overset{d}{=}\big(D_n^n,D_{n-1}^n,...,D_0^n\big)\qquad  \forall~n\geq 0.
	\feqn
	Under Assumption~\ref{ass1}, $X$ is transient to the right, and hence there exists an infinite sequence of times when the random walk
	moves forward to a ``fresh point", i.~e. to a site which has been never visited before and to which it will never return afterwards \cite{kosygina1,martin1}.
	It turns out \cite{bsingh1,kosygina1} that for the branching process $V$ this implies that the following random times
	are finite with probability one under the law $P_0^V:$
	\beqn
	\label{sigma}
	\sigma_{-1}:=0 \quad \mbox{and}\quad \sigma_k:=\inf\bigl\{i>\sigma_{k-1}: V_i = 0\bigr\},\quad k\geq 0.
	\feqn
	Thus $(\sigma_k)_{k\geq 0}$ is the sequence of renewal times in which the extinction occurs and
	the branching process starts afresh due to the immigration. Notice that while the immigrants serve as
	``founders of dynasties" of descendants, they themselves are not counted in the population of the branching process.
	In what follows we refer to the part of the branching process evolving  between two successive extinction times as a {\em life cycle}
	of the process. The difference $\sigma_k-\sigma_{k-1},$ $k\geq 0$, represents therefore the duration of the $(k+1)$-th life cycle.
	We remark that, although under our assumptions the event $\sigma_k-\sigma_{k-1}=1$ can happen with a positive probability for any $k\geq 0,$  the
	branching process does not get absorbed at zero, and is eventually revived in a future generation with a strictly positive number of immigrants. Let
	\beqn
	\label{rhom}
	\varrho_m := \min\bigl\{k\geq 0:  \sigma_k \geq m \bigr\}
	\feqn
	denote the index of the first renewal epoch at or after generation $m.$
	\par
	Further, let
	\beq
	S_k:=\sum_{i=\sigma_{k-1}}^{\sigma_k-1}V_i, \qquad k\geq 0,
	\feq
	be the total population present during the $(k+1)$-th cycle, and let
	\beq
	M_k:= \max_{\sigma_{k-1} \leq i <\sigma_k}V_i, \qquad k\geq 0,
	\feq
	be the size of the most populated generation in the $(k+1)$-th cycle.
	Notice that the sequence $(\sigma_k-\sigma_{k-1})_{k\geq 0}$ as well as the sequence of pairs $(S_k,M_k)_{k\geq 0}$ are i.i.d. under $P_0^V.$
	The following asymptotic results hold under Assumption~\ref{ass1}
	(see, for instance, Theorem~2.1, Theorem~2.2, and Lemma~8.1 in \cite{kosygina2}, respectively):
	\beqn
	\label{sum_law}
	&&\lim_{n\rightarrow\infty}{n^{\delta/2} P^V_0(S_0 > n)} = K_0 \in (0,\infty),
	\\
	\label{sigma_law}
	&&
	\lim_{n\rightarrow\infty}{n^\delta P^V_0(\sigma_0 > n)} = K_1 \in (0,\infty),
	\\
	\label{Mlemma}
	&&\lim_{n\rightarrow\infty}{n^{\delta} P^V_0(M_0 > n)} = K_2 \in (0,\infty).
	\feqn
	In a few instances, it will be convenient to consider the following stationary version of the branching process.
	Since the random walk $X_n$ is transient to the right, the Markov chain $V_n$ converges in law to a stationary distribution $V_\infty$
	which is the total number of steps left from the site $0$ to site $-1$ for the whole infinite duration of the random walk (see, for instance, \cite[p.~631]{bsingh1}). Notice that $V_\infty$ is finite with probability one because the walk is transient to the right. Let $Y_{k,n}$ denote the number of descendants in generation $n$ of the immigrant that arrived at time $k$.
	To be consistent with the description of the random walk, we enumerate the population in generation $n$ so that, whenever $k<\ell\leq n,$ every descendant of the generation-$\ell$ immigrant has a smaller label than every descendant of the generation-$k$ immigrant. Then,
	\beqn
	\label{vk}
	V_n=\sum_{k=0}^{n-1} Y_{k,n}.
	\feqn
	This labeling convention corresponds to the assumption that at each site $z\geq 0,$ the coins $B(z,i)$ are used by the random walk in the order
	determined by the second index $i.$ In particular, (cf. \eqref{hren}) \eqref{wmk} ensures that
	\beqn
	\label{hren1}
	\big(D^{(k+1)}_k,D^{(k+2)}_k-D^{(k+1)}_k,\ldots, D^{(n+1)}_k-D^{(n)}_k\big)\overset{d}{=} (Y_{n-k,n-k},Y_{n-k-1,n-k},\ldots,Y_{0,n-k})
	\feqn
	whenever $1\leq k<n.$
	\par
	One can introduce ``dummy" immigrants arriving at negative times (as it is done, for instance, in \cite{key-br, multi}) and consider the stationary process
	\beqn
	\label{station}
	\witi V_n=\sum_{k=-\infty}^n Y_{k,n}, \qquad n\in\zz,
	\feqn
	where $\witi V_n$ is distributed as $V_\infty$ for each $n\in\zz.$ The stationary version of the branching process is
	useful because
	\beqn
	\label{identity}
	\bigl(\witi V_{-n}\bigr)_{n\geq 1}\overset{d}{=}(D_n)_{n\geq 1},
	\feqn
	where $X\overset{d}{=}Y$ indicates that random variables (sequences in this specific case) $X$ and $Y$ have the same distribution, and for $k\geq 0,$
	\beqn
	\label{dikey}
	D_k:= \lim_{m\to\infty} D_k^m=\sum_{i=0}^\infty\one{X_{i+1}=k-1,\, X_i=k}
	\feqn
	is the total number of down-crossing steps of the ERW from site $k$ to $k-1.$
	\par
	For $n\leq 0$ and integer $k\in [n,0],$ let
	\beqn
	\label{zi}
	Z_k^n=\sum_{m=n}^k Y_{m,k}.
	\feqn
	Then, for any fixed $n\geq 0,$ (cf. \eqref{hren} and \eqref{hren1}):
	\beqn
	\label{zink}
	\big(D_n^n,D_{n-1}^n,...,D_0^n\big)\overset{d}{=}\big(Z_{-n}^{-n},Z_{-n+1}^{-n},...,Z_0^{-n}\big)\overset{d}{=}(V_0,\ldots,V_n).
	\feqn
	Most of the time, the identity in \eqref{zink} will be enough for our purposes. However, occasionally we will use
	the following stronger version:
	\beqn
	\label{zin1}
	\big(D_k^n\big)_{n\geq 0,\,0\leq k\leq n}\overset{d}{=}\big(Z_{-n+k}^{-n}\big)_{n\geq 0,\,0\leq k\leq n}.
	\feqn
	The last identity is the point at which one has to distinguish carefully between a fixed-level branching-process representation and a process-level representation.  We record the precise form of this passage, together with the renewal convention used later in the proof.
	\begin{lemma}[Triangular representation and renewal transfer]
		\label{triangular-renewal}
		Under Assumption~\ref{ass1}, the equality in distribution in \eqref{zin1} holds as an identity of the whole triangular array, not only row by row.  Equivalently, for every $N\in\nn,$
		\beq
		\big(D_k^n\big)_{0\leq n\leq N,\,0\leq k\leq n}\overset{d}{=}
		\big(Z_{-n+k}^{-n}\big)_{0\leq n\leq N,\,0\leq k\leq n}.
		\feq
		Furthermore, let $\tau=\inf\{j\geq 0:\witi V_j=0\}$ for the stationary process $\big(\witi V_j\big)_{j\in\zz}$ and let
		\beq
		\mathcal C_i=(\sigma_i-\sigma_{i-1},S_i,M_i)
		\feq
		be the successive cycle triples after a zero of the branching process.  The sequence of cycle triples strictly after the finite residual piece ending at $\tau$ has the same law as $\big(\mathcal C_i\big)_{i\geq 0}$ under $P_0^V.$ Consequently, if an event depending on a one-sided sequence of cycles is unchanged by deleting finitely many initial cycles, then its probability is the same whether the cycles are read from the stationary construction after the residual cycle or from the process $V$ under $P_0^V.$
	\end{lemma}
	\begin{proof}
		The first assertion follows from the coin construction.  For a fixed terminal level $n,$ the down-crossing vector is obtained by running the recursion
		\beq
		D_n^n=0,\qquad D_k^n=F_{D_{k+1}^n+1}^{(k)},\qquad 0\leq k<n,
		\feq
		with the failure-before-success variables determined by the stacks of coins at the sites $0,\ldots,n-1.$  If the same construction is written with the space-time labels shifted so that the terminal level $n$ is placed at time $0,$ the descendants of immigrants arriving at times $-n,-n+1,\ldots,0$ give exactly the variables $Z_{-n+k}^{-n}.$  For finitely many terminal levels $0\leq n\leq N,$ both sides are measurable functions of finitely many coin stacks and the same deterministic failure-before-success maps.  The i.i.d. assumption on the cookie piles then permits the common relabeling of the stacks, and hence gives the displayed finite-dimensional equality.  Since $N$ is arbitrary, \eqref{zin1} holds as an equality in distribution of triangular arrays.
		\par
		For the second assertion, observe that zeros of $\witi V$ are regeneration times.  The time $\tau$ is finite almost surely because $\witi P^V(\witi V_0=0)>0$ and the stationary chain visits zero with positive frequency.  By the strong Markov property at $\tau$ and by the independence of the cookie piles to the right of that regeneration point, the cycles following $\tau$ are independent and have the same common law as the cycles of the branching process started from zero.  The part of the stationary process before this first regeneration time is a single finite residual cycle.  Therefore deleting it, or deleting any other finite number of initial cycles, cannot change limsup, liminf, or infinitely-often events whose normalizing sequences tend to infinity and whose cycle windows tend to infinity.  This proves the transfer statement.
	\end{proof}
	We will denote by $\witi P_\omega^V$ and $\witi P^V,$ respectively, the quenched (in a fixed environment $\omega$) and annealed (in a random environment) law of the stationary branching process. Notice that
	\beqn
	\label{ground}
	P_0^V(\,\cdot\,)=\witi P^V\big(\cdot \big| \witi V_0=0\big).
	\feqn
	It follows, for instance from \eqref{identity}, that $\witi P^V\big(\witi V_0=0\big)>0,$ and in fact $\witi P_\omega^V\big(\witi V_0=0\big)>0$ with probability one.
	\par
	We shall use the notation $\sigma_k,$ $\varrho_k,$ $S_k,$ and $M_k$ also for cycles read from the stationary process, but with the convention made explicit in Lemma~\ref{triangular-renewal}.  The cycle intersecting the time origin in the stationary process is a finite residual cycle and need not have the same distribution as a fresh cycle under $P_0^V.$  After the first zero at or to the right of the origin, however, the future cycle triples are i.\,i.\,d. and have the same law as $(\sigma_0,S_0,M_0)$ under $P_0^V.$  Thus the almost-sure limsup, liminf, and infinitely-often events used below may be evaluated under $P_0^V$ after applying the triangular representation \eqref{zin1}, because adding or deleting this finite residual piece does not affect those events.
	\section{Proof of the main results}
	\label{proofs}
	We begin with a brief outline of the proof. First, using properties of the regular variation and
	the tail asymptotic of the renewal times which is given by \eqref{sigma_law}, we reduce the study of
	$\xi^*_n$ to that of $\xi^*_{T_n}.$ The bounds for the latter sequence, stated in \eqref{mineq}, enable us then to exploit
	the connection between the random walk and the branching process $V$ introduced in Section~\ref{reduction}.
	We remark that a similar strategy has been used for RWRE in \cite{nina1}. The implementation of this approach
	(which at most stages is technically fairly different in this paper from the one presented in \cite{nina1}) is based
	on the existence of the renewal structure (life cycles) for the branching process and the asymptotic results for
	the distributional tails of the key random variables stated in \eqref{sum_law}-\eqref{Mlemma} above.
	\par
	A technical point used several times below is that fixed-level identities such as \eqref{zink} are not, by themselves, enough for limsup or Borel-Cantelli arguments over infinitely many hitting levels.  Whenever such an event is transferred from the ERW to a branching process, the transfer is through the triangular-array identity and renewal convention of Lemma~\ref{triangular-renewal}.
	\par
	We start the proof with the following $0-1$ law for the maximal occupation time of the random walk.
	A similar statement for one-dimensional random walks in random environment is given in \cite[Proposition~3.1]{nina1}.
	\begin{lemma}
		\label{zero_one}
		Let $\phi:\nn \rightarrow \rr^+$ be an unbounded eventually increasing function such that, for every fixed integer $r,$
		\beq
		\label{finite-shift-phi}
		\lim_{n\to\infty}\frac{\phi(n+r)}{\phi(n)}=1.
		\feq
		Then
		\beq
		\limsup_{n\rightarrow \infty}\frac{\xi^*_n}{\phi(n)}=c_\phi\in [0,\infty] \qquad P_0-\as
		\feq
	\end{lemma}
	\begin{proof}
		Set
		\beq
		L(X):=\limsup_{n\to\infty}\frac{\xi^*_n}{\phi(n)},
		\feq
		where the local times are computed along the infinite path $X.$ We first show that, for every fixed $k\geq 0,$ the value of $L(X)$ is measurable with respect to the future half-line coin stacks $\sigma(B^{(k)},B^{(k+1)},\ldots).$  Let $N_k$ be the total number of times spent by the walk in $(-\infty,k-1].$  Under Assumption~\ref{ass1}, $N_k<\infty,$ $P_0$-almost surely.  If $X^{(k)}$ denotes the path obtained by deleting from $X$ all visits to $(-\infty,k-1],$ then the time parameter of $X^{(k)}$ differs from the original time parameter by the fixed finite amount $N_k$ once the last visit to $(-\infty,k-1]$ has occurred.  Moreover, the local times at sites below $k$ are bounded by $N_k,$ while the local times at sites at least $k$ are the local times read from $X^{(k)}$ after adding only finitely many initial visits.  Since $\phi(n)\to\infty$ and \eqref{finite-shift-phi} holds, these finite changes in time and in local times do not change the value of the limsup.  Consequently $L(X)$ is determined by $X^{(k)}.$
		\par
		Following \cite{martin1} (see the paragraph right before Lemma~3 there), define recursively a sequence of random times $(\eta_{k,m})_{m\in\nn,k\in\zz}$ by setting
		\beq
		\eta_{k,0}:=-1\qquad \mbox{and}\qquad \eta_{k,m+1}:=\inf\{n>\eta_{k,m}:X_n\geq k\}.
		\feq
		For a fixed $k\in \nn,$ the sequence $X^{(k)}=(X_{\eta_{k,m}})_{m\geq1}$ is the path obtained by retaining only the visits to $\{x\in\nn:x\geq k\}.$  Zerner's argument \cite[Lemma~3]{martin1} applies here because Assumption~\ref{ass1} implies transience to the right, and gives that the quenched law of $X^{(k)}$ in a fixed environment $\omega$ is independent of $(\omega_i)_{i\leq k-1}.$  In the coin construction this retained path is measurable with respect to
		\beqn
		\label{zheka}
		\sigma\bigl(B^{(k)},B^{(k+1)},B^{(k+2)},\ldots\bigr),\qquad k\geq0.
		\feqn
		It follows that, for $\pp$-almost every environment $\omega,$ the random variable $L(X)$ is a tail random variable of the independent sequence of coin stacks under $P_{0,\omega}.$  Kolmogorov's $0$-$1$ law therefore implies that $L(X)$ is $P_{0,\omega}$-almost surely equal to a deterministic value depending only on the environment; denote this value by $L(\omega).$
		\par
		The preceding paragraph holds for every $k.$  Hence, for every rational $q\geq0,$ the event $\{L(\omega)\leq q\}$ belongs to $\sigma(\omega_k,\omega_{k+1},\ldots)$ for every $k.$  Equivalently, it belongs to the tail $\sigma$-field of the i.i.d. sequence of cookie piles.  Kolmogorov's $0$-$1$ law for the environment implies that $\pp(L\leq q)\in\{0,1\}$ for every rational $q.$  The distribution function of $L$ is therefore a step function with a single jump, possibly at $\infty,$ and $L$ is $P_0$-almost surely constant.  This proves the lemma.
	\end{proof}
	
	\subsection{Non-ballistic regime: Proof of Theorem~\ref{thm2}}
	\label{prof-nb}
	\begin{proof} [\textbf{Part~(i)}]
		We first consider the case when $\delta\in (1,2).$ The remaining case $\delta=2$ is exceptional and is treated separately in Lemma~\ref{lemma4} below.
		\begin{lemma}
			\label{lemma1}
			Assumption~\ref{ass1} implies \eqref{positive} when $\delta \in (1,2)$.
		\end{lemma}
		\begin{proof}[Proof of Lemma~\ref{lemma1}]
			In order to prove the claim of the lemma, it suffices to show that for all $\delta\in (1,2)$ and any constant $c>0$
			which is small enough we have
			\beqn
			\label{m4}
			\liminf_{n\to \infty}{P_0\bigl(\xi^*_{T_n} \geq c \sqrt{T_n} \bigr)}>0.
			\feqn
			Indeed, \eqref{m4} together with the reverse Fatou's lemma imply that
			\beq
			&&
			P_0\Bigl( \limsup_{n\to\infty}{\frac{\xi^*_n}{\sqrt{n}}}>c\Bigr) \geq P_0\Bigl( \limsup_{m\to\infty}\frac{\xi^*_{T_m}}{T_m^{1/2}} > c  \Bigr)=
			E_0\Bigl( \limsup_{m\to\infty}\one{\xi^*_{T_m} > c T_m^{1/2}}  \Bigr)\\
			&&
			\qquad
			\geq \limsup_{m\to\infty} E_0\bigl(\one{\xi^*_{T_m} > c T_m^{1/2}}  \bigr)=\limsup_{m\to\infty} P_0\bigl(\xi^*_{T_m}> c T_m^{1/2} \bigr) > 0.
			\feq
			By virtue of Lemma~\ref{zero_one}, this yields the claim.
			\par
			We now turn to the proof that \eqref{m4} holds for any sufficiently small $c>0.$
			Observe that according to \eqref{tm_rep} and \eqref{mineq}, we have
			\beq
			P_0\bigl( (\xi^*_{T_m})^2 \geq c^2 T_m \bigr) &\geq & P_0\Bigl( \bigl(\max_{0\leq k\leq m} D_k^m \bigr)^2 \geq c^2 \Bigl(m+2\sum_{k=0}^m D_k^m\Bigr) \Bigr) \\
			&\geq& P_0^V\Bigl( \max_{0\leq i<\varrho_m} M_i^2 \geq c^2 \Bigl(m+2\sum_{i=0}^{\varrho_m}{S_i}\Bigr)\Bigr).
			\feq
			Denote by $\mu:=E_0^V[\sigma_0]$ the annealed (i.~e., taken under the law $P_0^V$) expectation of $\sigma_0,$
			choose an arbitrary $\delta'\in(1,\delta),$ and define
			\beqn
			\label{a-m-def}
			a^{\pm}_m=\bigl\lfloor\mu^{-1} (m \pm m^{1/\delta'})\bigr\rfloor,
			\feqn
			where we use the notation $\lfloor x \rfloor$ to denote the integer part of a real number $x.$  We shall also write
			\beq
			a_m=\lfloor \mu^{-1}m\rfloor
			\feq
			when no upper or lower deviation buffer is needed.
			It follows from the last inequality stated above that
			\beqn
			\label{auxi}
			&&
			P_0\bigl( (\xi^*_{T_m})^2 \geq c^2 T_m \bigr)
			\\
			&&
			\qquad
			\geq P_0^V\Bigl( \max_{1\leq i\leq a_m^-}{M_i^2} \geq c^2\Bigl( m + 2\sum_{i=1}^{a_m^++1}S_i\Bigr) \Bigr)
			- P_0^V(\varrho_m >a^+_m) - P_0^V(\varrho_m < a^-_m).
			\nonumber
			\feqn
			Since $\{\varrho_m>a_m^+\}$ implies $\{\sigma_{a_m^+}<m\}$, and since $\sigma_{a_m^+}=\sum_{i=0}^{a_m^+}(\sigma_i-\sigma_{i-1})$, we have
			\beqn
			\label{an-ineql}
			P_0^V(\varrho_m > a^+_m) \leq P_0^V\Bigl( \frac{\sum_{i=0}^{a_m^+}(\sigma_{i}-\sigma_{i-1}-\mu) }{(a^+_m)^{1/\delta}}
			\leq - \frac{m^{1/\delta'}}{(a^+_m)^{1/\delta}}\Bigr).
			\feqn
			Hence \eqref{sigma_law} and a stable limit theorem for i.i.d. random variables $\sigma_{i}-\sigma_{i-1}$ (see, for instance,
			Theorem~1.5.1 in \cite{rwbook}) imply that $\lim_{m\to\infty} P_0^V(\varrho_m >a^+_m)=0.$ Similarly, one can show that $\lim_{m\to\infty} P_0^V(\varrho_m < a^-_m) =0.$
			Therefore, in order to prove \eqref{m4}, it suffices to show that the following strict lower bound holds:
			\beqn
			\label{lem1ine7}
			\liminf_{m\to\infty} P_0^V\Bigl( \max_{1\leq i\leq a_m^-}{M_i^2} \geq c^2\Bigl( m + 2\sum_{i=1}^{a_m^++1}S_i\Bigr) \Bigr)>0.
			\feqn
			Toward this end, recall $K_0$ and $K_2$ from \eqref{sum_law} and \eqref{Mlemma} respectively, fix any positive constants $\beta,\gamma>0$ such that
			\beqn
			\label{tend}
			\frac{1}{2}\cdot \mu^{-1}\bigl(K_2\cdot \beta^{-\delta/2}-K_0\cdot \gamma^{-\delta/2}\bigr)-K_2^2\cdot \beta^{-\delta}>0,
			\feqn
			and observe that
			\beq
			&&
			P_0^V\Bigl(\max_{1\leq i\leq a_m^-}{M_i^2} \geq c^2\Bigl( m+ 2\sum_{i=1}^{a_m^++1}S_i \Bigr)\Bigr) \geq
			\\
			&& \qquad  \geq
			P_0^V\Bigl( \max_{1\leq i\leq a_m^-}{M_i^2} \geq \beta m^{2/\delta}, \, m + 2\sum_{i=1}^{a_m^++1}S_i \leq \frac{\beta m^{2/\delta}}{c^2} \Bigr)
			\geq P_0^V\Bigl(\bigcup_{i=1}^{a_m^-} A_{i,m}\Bigr),
			\feq
			where
			\beqn
			\label{aist}
			A_{i,m}:=\Bigl\{M_i^2\geq \beta m^{2/\delta},\,\, S_i\leq \gamma m^{2/\delta},\,\,
			m+2\sum_{\overset{1\leq j\leq a_m^++1,}{j\neq i}}S_j < (\beta/c^2-\gamma) m^{2/\delta} \Bigr\}.
			\feqn
			Therefore, the inclusion-exclusion formula yields
			\beqn
			\label{more}
			&&
			P_0^V\Bigl( \max_{1\leq i\leq a_m^-}{M_i^2} \geq c^2 \Bigl(m + 2\sum_{i=1}^{a_m^++1}S_i \Bigr)\Bigr) \geq P_0^V\Bigl(\bigcup_{i=1}^{a_m^-} A_{i,m}\Bigr)
			\nonumber
			\\
			&&
			\qquad
			\geq \sum_{i=1}^{a_m^-}P_0^V(A_{i,m})- \sum_{i=1}^{a_m^-}\sum_{j=i+1}^{a_m^-} P_0^V\Bigl(A_{i,m}\bigcap A_{j,m}\Bigr)
			\nonumber
			\\
			&&
			\qquad
			\geq  a_m^-P_0^V(A_{1,m})-(a_m^-)^2 P_0^V\bigl(M_1^2\geq \beta m^{2/\delta},\,M_2^2\geq \beta m^{2/\delta}\bigr).
			\feqn
			Using the independence of the life-cycles of the underlying branching process, we obtain that
			\beq
			P_0^V(A_{1,m}) =P_0^V\bigl(M_1^2\geq \beta m^{2/\delta},\, S_1\leq \gamma m^{2/\delta}\bigr)\cdot P_0^V\Bigl(m+2\sum_{j=2}^{a_m^++1}S_j <
			(\beta/c^2-\gamma)m^{2/\delta} \Bigr).
			\feq
			Taking into account \eqref{sum_law}, \eqref{Mlemma}, and \eqref{a-m-def}, one can deduce from the following inequality:
			\beq
			P_0^V(M_1^2\geq \beta m^{2/\delta},\, S_1\leq \gamma m^{2/\delta}) \geq P_0^V\bigl( M_1^2 \geq \beta m^{2/\delta} \bigr) -
			P_0^V\bigl(S_1> \gamma m^{2/\delta} \bigr),
			\feq
			that
			\beq
			\liminf_{m\to\infty} a_m^- \cdot P_0^V\bigl(M_1^2\geq \beta m^{2/\delta},\, S_1\leq \gamma m^{2/\delta}\bigr)
			\geq \mu^{-1}\bigl(K_2\cdot \beta^{-\delta/2}-K_0\cdot \gamma^{-\delta/2}\bigr).
			\feq
			Furthermore, it follows from \eqref{sum_law} and a stable limit theorem for i.i.d. variables $S_i$ (see, for instance, Theorem~1.5.1 in \cite{rwbook})
			that the following limit exists and is strictly positive:
			\beq
			\lambda(c,\beta,\gamma):=\lim_{m\to\infty} P_0^V\Bigl(m+2\sum_{j=2}^{a_m^++1}S_j < (\beta/c^2 -\gamma) m^{2/\delta} \Bigr) > 0.
			\feq
			Moreover, given $\beta,\gamma>0,$ we can choose $c>0$ so small that $\lambda(c,\beta,\gamma)>1/2.$
			For such a constant $c>0,$ \eqref{more} along with \eqref{Mlemma} yield
			\beq
			&& \liminf_{m\to\infty}P_0^V\Bigl( \max_{1\leq i\leq a_m^-}{M_i^2} \geq c^2 \Bigl(m + 2\sum_{i=1}^{a_m^++1}S_i \Bigr)\Bigr)
			\\
			&&
			\qquad
			\geq
			\lambda(c,\beta,\gamma)\cdot \mu^{-1}\bigl(K_2\cdot \beta^{-\delta/2}-K_0\cdot \gamma^{-\delta/2}\bigr)-K_2^2\cdot \beta^{-\delta}
			\\
			&&
			\qquad
			\geq
			\frac{1}{2}\cdot \mu^{-1}\bigl(K_2\cdot \beta^{-\delta/2}-K_0\cdot \gamma^{-\delta/2}\bigr)-K_2^2\cdot \beta^{-\delta}.
			\feq
			In view of \eqref{tend}, this implies \eqref{lem1ine7}, and hence \eqref{m4} for the indicated above values of
			the parameter $c>0.$
		\end{proof}
		Lemma~\ref{lemma1} yields part~(i) of Theorem~\ref{thm2} for $\delta\in(1,2).$
		The remaining case $\delta=2$ is considered in the next lemma, by using a slightly different approach which
		is based on a modification of an argument in \cite{nina1}, cf. \cite[p.~172]{nina1}.
		\begin{lemma}
			\label{lemma4}
			Let Assumption~\ref{ass1} hold and suppose that $\delta=2.$ Then \eqref{positive} holds
			along random subsequence $(T_m)_{m\geq 1}$ of integers $n.$ That is,
			\beq
			\limsup_{m\to\infty}\frac{\xi^*_{T_m}}{T_m^{1/2}} > 0, \qquad P_0-\as
			\feq
		\end{lemma}
		\begin{proof}[Proof of Lemma~\ref{lemma4}]
			Recall \eqref{rhom} and \eqref{a-m-def}. The proof of the lemma relies on the Borel-Cantelli argument which is
			applied to the following version of \eqref{auxi}:
			\beqn
			\label{auxi4}
			P_0\bigl( (\xi^*_{T_m})^2 \geq c^2 T_m \io\bigr)&\geq& P_0^V\Bigl( \max_{1\leq i\leq a_m^-}{M_i^2} \geq c^2\Bigl( m + 2\sum_{i=1}^{a_m^++1}S_i\Bigr)  \io \Bigr)
			\nonumber
			\\
			&&
			~- P_0^V(\varrho_m >a^+_m \io) - P_0^V(\varrho_m < a^-_m \io)
			\feqn
			with a suitably chosen subsequence of the integer indices $m.$ The inequality in \eqref{auxi4} is a process-level statement.  We first use the triangular representation of Lemma~\ref{triangular-renewal}, rather than only the fixed-row identity \eqref{zink}, and then use the renewal-transfer part of the same lemma to evaluate the resulting finite-shift invariant i.o. event under $P_0^V.$
			\par
			First, we will establish certain large deviation type estimates for the distributional tails of the random variables $\varrho_m.$
			To this end, observe that the inequality stated in \eqref{an-ineql} remains true for any $\delta>1,$ in particular for $\delta=2.$
			Furthermore, \eqref{sigma_law} and the large deviation estimate stated, for instance, in \cite[Theorem~3.4.1]{rwbook} imply that
			the following holds under Assumption~\ref{ass1} (with arbitrary $\delta>1$) for a suitable constant $c_1=c_1(\delta)>0:$
			\beqn
			\label{an-ineq3}
			P_0^V(\varrho_m >a^+_m) \leq \frac{c_1 a_m^+}{m^{\delta/\delta'} }.
			\feqn
			Similarly, using the following inequality instead of \eqref{an-ineql}:
			\beqn
			\label{endi}
			P_0^V(\varrho_m < a^-_m)\leq P_0^V(m \leq \sigma_{a_m^-})\leq P_0^V\Bigl( \frac{\sum_{i=0}^{a_m^-}(\sigma_{i}-\sigma_{i-1}-\mu) }{(a^-_m)^{1/\delta}} \geq
			\frac{m^{1/\delta'}}{(a^-_m)^{1/\delta}}\Bigr),
			\feqn
			one can deduce from \eqref{sigma_law} and \cite[Theorem~3.4.1]{rwbook} that the following holds under Assumption~\ref{ass1} (with arbitrary $\delta>1$)
			for some constant $c_2=c_2(\delta)>0:$
			\beqn
			\label{an-ineq4}
			P_0^V(\varrho_m < a^-_m) \leq \frac{c_2 a_m^-}{m^{\delta/\delta'}}.
			\feqn
			\begin{remark}
				We remark that, in the course of proving \eqref{an-ineq3}, in order to formally meet the lower tail conditions of Theorem~3.4.1 in \cite{rwbook} one
				can, for instance, use in \eqref{an-ineql} the following ``unpolarized" version of $\sigma_i-\sigma_{i-1}-\mu$ which has
				the same structure of upper and lower distribution tails:
				\beq
				(\sigma_i-\sigma_{i-1}-\mu)':=U_i\cdot (\sigma_i-\sigma_{i-1}-\mu),
				\feq
				where $U=(U_i)_{i\geq 0}$ is a sequence of i.i.d. Bernoulli random variables, independent of ``anything else"
				(i.~e., such that the probability law $P^U$ of $U$ is independent of the measure $P^V$ in the enlarged probability space), and such that
				\beq
				P^U(U_i=1)=P^U(U_i=-1)=\frac{1}{2}.
				\feq
				Note that (cf. \eqref{an-ineql}) $P_0^V\Bigl(\sum\limits_{i=0}^{a_m^+}(\sigma_{i}-\sigma_{i-1}-\mu) \leq - m^{1/\delta'} \Bigr)
				\leq
				P_0^V\Bigl(\sum\limits_{i=0}^{a_m^+}(\sigma_{i}-\sigma_{i-1}-\mu)' \leq - m^{1/\delta'} \Bigr).$
			\end{remark}
			Fix now any $c>0$ and recall \eqref{auxi}. For $k\in \nn,$ let
			\beq
			m_k=k^k,
			\feq
			and define the following events:
			\beq
			B_k&=&\Bigl\{\max_{a_{2m_{k-1}}< i\leq a_{m_k}^-}{M_i^2} \geq c^2\Bigl( m_k + 2\sum_{i=1}^{a_{m_k}^++1}S_i\Bigr)\Bigr\},\\
			D_k&=&\Bigl\{2c^2\cdot\sum_{i=1}^{a_{m_{k-1}}^++1}S_i>m_k\log m_k\Bigr\},\\
			E_k&=&\Bigl\{\max_{a_{2m_{k-1}}< i\leq a_{m_k}^-}{M_i^2} \geq 2 m_k\log m_k,\,c^2\Bigl( m_k + 2\sum_{j=a_{m_{k-1}}^++2}^{a_{m_k}^++1}S_j\Bigr)\leq m_k\log m_k\Bigr\}.
			\feq
			By virtue of \eqref{sum_law}, Theorem~3.8.1 along with Theorem~1.5.1-(ii) in \cite{rwbook} (notice that the random centering
			required in the former is given by the latter), there exists a positive constant $c_3=c_3(c)>0$ such that
			\beq
			P_0^V(D_k)\leq c_3\frac{m_{k-1}}{m_k\log m_k}\leq c_3 \frac{(k-1)^{k-1} }{k^{k+1}\log k}\leq c_3  k^{-2}.
			\feq
			Thus by the first Borel-Cantelli lemma, $P_0^V(D_k\io)=0.$ On the other hand, similarly to \eqref{aist} and \eqref{more}, setting
			\beq
			\witi A_{i,k}:=\Bigl\{M_i^2\geq 2 m_k \log m_k,\,\, c^2\Bigl(m_k+2\sum_{j=a_{m_{k-1}}^++2}^{a_{m_k}^++1}S_j\Bigr) \leq m_k \log m_k \Bigr\},
			\feq
			Choose any $i_k$ in the deterministic block $\{i:a_{2m_{k-1}}<i\leq a_{m_k}^-\}$ and use the exchangeability of the i.i.d. cycle triples inside this block.  We obtain
			\beq
			&&
			P_0^V( E_k ) \geq P_0^V\Bigl(\bigcup_{i=1+a_{2m_{k-1}}}^{a_{m_k}^-} \witi A_{i,k}\Bigr)
			\geq \sum_{i=1+a_{2m_{k-1}}}^{a_{m_k}^- }P_0^V\bigl(\witi A_{i_k,k}\bigr)
			- \sum_{i=1}^{a_{m_k}^-}\sum_{j=i+1}^{a_{m_k}^-} P_0^V\Bigl(\witi A_{i,k}\bigcap \witi A_{j,k}\Bigr)
			\\
			&&
			\qquad
			\geq  (a_{m_k}^- -a_{2m_{k-1}})P_0^V\bigl(\witi A_{i_k,k}\bigr)-(a_{m_k}^-)^2 P_0^V\bigl(M_1^2\geq m_k \log m_k,\,M_2^2\geq  m_k \log m_k\bigr).
			\feq
			Thus by virtue of \eqref{Mlemma} and Theorem~1.5.1-(ii) in \cite{rwbook}, if $c>0$ is sufficiently small, then
			\beq
			P_0^V(E_k) \geq \frac{c_4}{\log m_k}= \frac{c_4}{k \log k},
			\feq
			for a suitable constant $c_4>0.$ Since the events $E_k$ are independent of each other, the second Borel-Cantelli lemma yields $P_0^V(E_k\io)=1.$
			\par
			For an event $A$ in the underlying probability space, let $A^c$ denote its complement.
			Observe now that
			\beq
			E_k\bigcap D_k^c \subset B_k,\qquad k\geq 0,
			\feq
			and hence $P_0^V(B_k\io)=1$ for any constant $c>0$ small enough. Finally, \eqref{an-ineq3} and \eqref{an-ineq4} imply that
			\beq
			P_0^V(\varrho_{m_k} < a^-_{m_k}\io)=P_0^V(\varrho_{m_k} > a^+_{m_k}\io)=0.
			\feq
			It thus follows from \eqref{auxi4} that
			\beq
			P_0\bigl( \xi^*_n \geq c\sqrt{n}  \io\bigr)\geq P_0\bigl( (\xi^*_{T_m})^2 \geq c^2 T_m \io\bigr)=1
			\feq
			for sufficiently small values of the constant $c>0.$ This completes the proof of the lemma.
		\end{proof}
		In view of Lemmas~\ref{lemma1} and~\ref{lemma4}, the proof of part~(i) of Theorem~\ref{thm2} is completed.
	\end{proof}
	$\mbox{}$
	\begin{proof} [\textbf{Part~(ii)}]
		First, we will show that for any constant $\alpha>1/\delta,$
		\beq
		\lim_{n\to\infty}\frac{\xi_n^*}{n^{1/2}(\log  n)^\alpha}=0,\qquad P_0-\as
		\feq
		Fix any $\alpha>1/\delta$ and let $\phi(n)=n^{1/2}(\log n)^\alpha.$  Recall the alternative notation $\xi^*(n)$ for $\xi^*_n.$
		Observe that in order to prove the above claim, it suffices to show that
		\beqn \label{claim1}
		\limsup_{m\to\infty}\frac{\xi^*(T_m)}{\phi(T_{m-1})} = 0,\qquad P_0-\as
		\feqn
		Indeed, let $k_m,$ $m\in\nn,$ be the (uniquely defined) non-negative random integers such that
		\beqn
		\label{k-n-1}
		T_{k_m} < m \leq T_{k_m+1}, \qquad m\in\nn.
		\feqn
		Then, since $\phi(n)$ is an eventually increasing sequence,
		\beqn
		\label{26a}
		\limsup_{m\to\infty}\frac{\xi^*(m)}{\phi(m)} &\leq& \limsup_{m\to\infty}\frac{\xi^*(T_{k_m+1})}{\phi(m)} \leq \limsup_{m\to\infty}\frac{\xi^*(T_{k_m+1})}{\phi(T_{k_m})} \nonumber \\
		&\leq& \limsup_{m\to\infty}\frac{\xi^*(T_{m})}{\phi(T_{m-1})} = 0, \qquad  P_0-\as
		\feqn
		Recall now \eqref{an-ineq3} and \eqref{an-ineq4}.
		Let $\psi_\veps(m)=m^{2/\delta}(\log  m)^{-\veps}$ with $\veps\in(0,2\alpha)$ and let $m_i=3^i$ for $i\in\nn.$ Then, for any constant $c>0,$ we have
		\beqn
		\label{variant}
		&& \sum_{i=1}^\infty P_0\bigl(T_{m_{i}} < c\psi_\veps(m_{i+1})\bigr) \leq \sum_{i=1}^\infty
		\Bigl[ P_0^V\Bigl(2\sum_{k=0}^{a_{m_i}^--1} S_k < c\psi_\veps(m_{i+1})\Bigr) + P_0^V\bigl(\varrho_{m_i} < a_{m_i}^-\bigr)\Bigr]
		\nonumber
		\\
		&&\qquad \leq \sum_{i=1}^\infty \Bigl[ P_0^V\Bigl(2\max_{0\leq k <a_{m_i}^-} S_k < c\psi_\veps(m_{i+1})\Bigr) +
		P_0^V\bigl(\varrho_{m_i} < a_{m_i}^-\bigr)\Bigr]
		\\
		&&\qquad  \leq \sum_{i=1}^\infty \Bigl[ \Bigl( 1- \frac{c_5}{(c \psi_\veps(m_{i+1}))^{\delta/2}} \Bigr)^{a_{m_i}^-}
		+ \frac{c_2 a_{m_i}^-}{(m_i)^{\delta/\delta'}} \Bigr]
		\leq \sum_{i=1}^\infty c_6\cdot\Bigl[ e^{-c_7(i+1)^{\veps\delta/2}}+ e^{-c_8 i}\Bigr]<\infty,
		\nonumber
		\feqn
		where $c_5,c_6,c_7,c_8>0$ are suitable positive constants and $c_2$ is the constant which appears at \eqref{an-ineq4}.
		Since the preceding summability holds for every fixed $c>0,$ the Borel-Cantelli lemma, applied for instance to $c=1,2,\ldots,$ gives
		\beq
		\liminf_{i\to\infty} \frac{T_{m_i}}{\psi_\veps(m_{i+1})} = \infty, \qquad P_0-\as
		\feq
		Since for each $n\in \nn,$
		\beqn
		\label{lsz3}
		m_i< n \leq m_{i+1}\quad \mbox{for some $i\in\nn$ which is uniquely determined by $n,$}
		\feqn
		then
		\beqn
		\label{faccord}
		\liminf_{n\to\infty}\frac{T_n}{\psi_\veps(n)} \geq \liminf_{i\to\infty} \frac{T_{m_i}}{\psi_\veps(m_{i+1})} = \infty, \qquad P_0-\as,
		\feqn
		and hence $P_0\bigl( T_m < \psi_\veps(m) \io \bigr)=0.$
		\par
		Recall $D_k$ from \eqref{identity} and \eqref{dikey}. The pathwise bound \eqref{fineq}, the identity \eqref{identity}, and the inequality $D_k^m\leq D_k$ imply that for any constant $b>0,$
		\beqn
		\label{lsz1}
		&& P_0\Bigl( \limsup_{m\to\infty}\frac{\xi^*(T_m)}{\phi(T_{m-1})} > b \Bigr) \nonumber
		\\
		&&
		\qquad
		\leq P_0\Bigl( \limsup_{m\to\infty}\frac{\max_{0\leq k\leq m} D_k^m}{\phi(\psi_\veps(m-1))}>\frac{b}{2}\bigr)
		\leq P_0\Bigl( \limsup_{m\to\infty}\frac{\max_{0\leq k\leq m} D_k}{\phi(\psi_\veps(m-1))}>\frac{b}{2}\bigr) \nonumber
		\\
		&&
		\qquad
		\leq P_0^V\Bigl( \limsup_{m\to\infty}\frac{\max_{0\leq k\leq \varrho_m} M_k}{\phi(\psi_\veps(m-1))}>\frac{b}{2}\Bigr)
		\leq  P_0^V\Bigl( \limsup_{m\to\infty}\frac{\max_{0\leq k\leq m} M_k}{\phi(\psi_\veps(m-1))}>\frac{b}{2}\Bigr)\nonumber
		\\
		&&
		\qquad
		=P_0^V\Bigl( \limsup_{m\to\infty}\frac{\max_{0\leq k\leq m} M_k}{m^{1/\delta} (\log m)^{\alpha-\veps/2}}>\frac{b\delta^\alpha}{2^{1+\alpha}}\Bigr),
		\feqn
		where in the last but one step we used the inequality $\varrho_m\leq m.$ The passage from the ERW down-crossing field to the cycle maxima in \eqref{lsz1} uses Lemma~\ref{triangular-renewal}.  The fixed-$m$ identity \eqref{zink} would not be sufficient here, because the event contains a limsup over $m.$
		\par
		Let $r_m\in [0,m]$ be the largest integer such that $\max_{0\leq k\leq m} M_k=M_{r_m}.$
		That is
		\beqn
		\label{gm}
		r_m=\max\bigl\{n \in [0,m]: M_n=\max_{0\leq k\leq m} M_k \bigr\},\qquad m\in\nn.
		\feqn
		Note that $r_m\leq m$ for any $m\in\nn.$ Therefore,
		\beqn
		\label{lsz4}
		&&
		P_0^V\Bigl( \limsup_{m\to\infty}\frac{\max_{0\leq k\leq m} M_k}{m^{1/\delta}  (\log m)^{\alpha-\veps/2}}>\frac{b\delta^\alpha}{2^{1+\alpha}}\Bigr)
		=P_0^V\Bigl( \limsup_{m\to\infty}\frac{M_{r_m}}{m^{1/\delta}  (\log m)^{\alpha-\veps/2}}>\frac{b\delta^\alpha}{2^{1+\alpha}}\Bigr)\nonumber
		\\
		&&
		\qquad
		\leq
		P_0^V\Bigl( \limsup_{m\to\infty}\frac{M_{r_m}}{r_m^{1/\delta}  (\log r_m)^{\alpha-\veps/2}}>\frac{b\delta^\alpha}{2^{1+\alpha}}\Bigr)
		\nonumber
		\\
		&&
		\qquad
		\leq
		P_0^V\Bigl( \limsup_{m\to\infty}\frac{M_m}{m^{1/\delta}  (\log m)^{\alpha-\veps/2}}>\frac{b\delta^\alpha}{2^{1+\alpha}}\Bigr).
		\feqn
		It follows from \eqref{Mlemma} that if $\veps\in (0,2\alpha)$ is chosen in such a way that in fact $\alpha-\veps/2>1/\delta,$ then
		\beq
		\sum_{m=1}^\infty P_0^V\Bigl( \frac{M_m}{m^{1/\delta}  (\log m)^{\alpha-\veps/2}}>\frac{b\delta^\alpha}{2^{1+\alpha}}\Bigr)  <\infty.
		\feq
		Thus the Borel-Cantelli lemma combined with \eqref{lsz3} implies that for $\veps\in (0,2\alpha-2/\delta),$
		\beq
		P_0^V\Bigl( \frac{M_m}{m^{1/\delta}  (\log m)^{\alpha-\veps/2}}>\frac{b\delta^\alpha}{2^{1+\alpha}}\io\Bigr)=0.
		\feq
		It then follows from \eqref{lsz1} and \eqref{lsz4} that
		\beq
		P_0\Bigl( \limsup_{m\to\infty}\frac{\xi^*(T_m)}{\phi(T_{m-1})} > b \Bigr)
		\leq  P_0^V\Bigl( \limsup_{m\to\infty}\frac{\max_{0\leq k\leq m} M_k}{\phi(\psi_\veps(m-1))}>\frac{b}{2}\Bigr)=0.
		\feq
		This completes the proof of \eqref{claim1}.
		\\
		$\mbox{}$
		\\
		We now turn to the proof that for any constant $\alpha>1/\delta,$
		\beq
		\lim_{n\to\infty}\frac{(\log {n})^{\alpha} \xi_n^*}{n^{1/2}}=\infty, \qquad P_0-\as
		\feq
		Fix any $\alpha>1/\delta$ and let $\phi(n)=n^{1/2}(\log n)^{-\alpha}.$ In order to prove the claim, it suffices to show that
		\beqn
		\label{claim3}
		\liminf_{m\to\infty}\frac{\xi^*(T_m)}{\phi(T_{m+1})} = \infty, \qquad P_0-\as
		\feqn
		Indeed, in view of \eqref{k-n-1}, it follows from \eqref{claim3} that
		\beq
		\infty = \liminf_{m\to\infty}\frac{\xi^*(T_m)}{\phi(T_{m+1})} &\leq& \liminf_{m\to\infty}\frac{\xi^*(T_{k_m})}{\phi(T_{k_m+1})} \leq \liminf_{m\to\infty}\frac{\xi^*(m)}{\phi(m)}, \qquad  P_0-\as
		\feq
		We will next show that \eqref{claim3} indeed holds true. Fix any $\veps\in (0,\alpha\delta-1)$ and define
		\beq
		\psi_\veps(m):=m^{2/\delta}(\log  m)^{2/\delta+\veps},\qquad m\in \nn.
		\feq
		By a counterpart of the law of the iterated logarithm for i.i.d. random variables in the domain of attraction of a stable law
		and with infinite variance, we have (see, for instance, Theorem~1.6.6. in \cite{rwbook}):
		\beq
		\limsup_{m\to\infty}\frac{\sum_{k=1}^m S_k}{\psi_\veps(m)} = 0, \qquad P_0^V-\as
		\feq
		Combining this result with the fact that $P_0^V\bigl(\lim_{m\to\infty}\frac{\varrho_m}{m}=\mu^{-1}\bigr)=1$
		(which is an implication of the renewal theorem applied to the renewal sequence $\sigma_k$), and using the triangular representation of Lemma~\ref{triangular-renewal} for this process-level bound, we obtain that
		\beq
		P_0\bigl(T_m > \psi_\veps(m) \io \bigr)=0.
		\feq
		Therefore,
		\beq
		\liminf_{m\to\infty}\frac{\xi^*(T_m)}{\phi(\psi_\veps(m+1))} \leq \liminf_{m\to\infty}\frac{\xi^*(T_m)}{\phi(T_{m+1})},\qquad P_0-\as
		\feq
		Thus it suffices to prove that the left-hand side of the above inequality is infinity. Recall \eqref{a-m-def} and \eqref{an-ineq4}.
		Then, for $m_i=3^i$  and any constant $c>0,$ similarly to \eqref{variant}, we have:
		\beqn
		\label{44a}
		&& \sum_{i=1}^\infty P_0\Bigl( \xi^*(T_{m_i}) < c \phi\bigl(\psi_\veps(m_{i+1}+1)\bigr) \Bigr)
		\leq \sum_{i=1}^\infty P_0^V\Bigl( \max_{0\leq k<\varrho_{m_i}} M_k < c \phi\bigl(\psi_\veps(m_{i+1}+1)\bigr) \Bigr)
		\nonumber \\
		&& \qquad \leq \sum_{i=1}^\infty \Bigl[P_0^V\Bigl( \max_{0\leq k<a_{m_i}^-} M_k <
		c \phi\bigl(\psi_\veps(m_{i+1}+1)\bigr) \Bigr) + P_0^V(a_{m_i}^->\varrho_{m_i})\Bigr]
		\nonumber
		\\
		&& \qquad \leq \sum_{i=1}^\infty c_9\cdot \Bigl[e^{-c_{10}(i+1)^{\alpha\delta-1-\veps\delta/2}} +  \frac{a_{m_i}^-}{(m_i)^{\delta/\delta'}}\Bigr] <\infty ,
		\feqn
		where $c_9>0$ and $c_{10}>0$ are some appropriate positive constants. Therefore, the Borel-Cantelli lemma yields
		(recall that $\delta\leq 2$ and the value of the parameter $\veps$ is chosen from the interval
		$(0,\alpha\delta-1),$ so that $\alpha\delta-1-\veps\delta/2\geq \alpha\delta-1-\veps>0$):
		\beq
		\liminf_{i\to\infty}\frac{\xi^*(T_{m_i})}{\phi(\psi_\veps(m_{i+1}+1))} =\infty,\qquad P_0-\as
		\feq
		This completes the proof of \eqref{claim3} by using a suitable variation of \eqref{faccord}.
	\end{proof}
	\subsection{Ballistic regime: Proof of Theorem~\ref{thm1}}
	\label{ballistic}
	By \cite[Theorem~1.5.3]{rv-book} combined with \cite[Theorem~1.9.5]{rv-book}, the regularly varying sequence
	$\phi(n)$ is asymptotically equivalent to a non-decreasing sequence $\bar \phi(n)$
	which is also regularly varying with index $\rho.$ By the asymptotic equivalence we
	mean that $\lim_{n\to\infty} \bar \phi(n)/\phi(n)=1.$ Since
	\beq
	\limsup_{n\to\infty} \frac{\xi_n^*}{\phi(n)}=\limsup_{n\to\infty} \frac{\xi_n^*}{\bar \phi(n)}
	\qquad \mbox{and} \qquad
	\liminf_{n\to\infty} \frac{\xi_n^*}{\phi(n)}=\liminf_{n\to\infty} \frac{\xi_n^*}{\bar \phi(n)},
	\feq
	we can assume without loss of generality that $\phi(n)$ is a non-decreasing sequence. In the rest
	of the paper we will make this additional assumption without further notice.
	\begin{proof} [\textbf{Part~(i)}]
		Suppose first that
		\beqn
		\label{phi1}
		\sum_{m=1}^\infty{\frac{1}{[\phi(m)]^{\delta}}} < \infty.
		\feqn
		Recall $D_k$ from \eqref{identity} and \eqref{dikey}. Using \eqref{fineq} and the inequalities $\varrho_m\leq m,$ $m\leq T_m,$ and $D_k^m\leq D_k,$ we obtain that for any constant $b>0,$
		\beqn
		\label{lsz}
		&& P_0\Bigl( \limsup_{m\to\infty}\frac{\xi^*(T_m)}{\phi(T_{m-1})} > b \Bigr) \nonumber
		\\
		&&
		\qquad
		\leq P_0\Bigl( \limsup_{m\to\infty}\frac{\max_{0\leq k\leq m} D_k^m}{\phi(m-1)}>\frac{b}{2}\bigr) \leq
		P_0\Bigl( \limsup_{m\to\infty}\frac{\max_{0\leq k\leq m} D_k}{\phi(m-1)}>\frac{b}{2}\bigr)
		\nonumber \\
		&&
		\qquad
		\leq P_0^V\Bigl( \limsup_{m\to\infty}\frac{\max_{0\leq k\leq \varrho_m} M_k}{\phi(m-1)}>\frac{b}{2}\Bigr)
		\leq  P_0^V\Bigl( \limsup_{m\to\infty}\frac{\max_{0\leq k\leq m} M_k}{\phi(m-1)}>\frac{b}{2}\Bigr). \qquad \qquad \mbox{}
		\feqn
		The second inequality in \eqref{lsz} is another use of Lemma~\ref{triangular-renewal}.  The bound involves the limiting down-crossing field and a limsup over $m,$ so the row-wise identity \eqref{zink} alone is not enough.
		\par
		Recall $r_m$ from \eqref{gm}. Then (recall that the sequence $\phi(m)$ is regularly varying, and hence $\lim_{m\to\infty}\frac{\phi(m)}{\phi(m-1)}=1$),
		\beq
		&&
		P_0^V\Bigl( \limsup_{m\to\infty}\frac{\max_{0\leq k\leq m} M_k}{\phi(m-1)}>\frac{b}{2}\Bigr)
		=P_0^V\Bigl( \limsup_{m\to\infty}\frac{M_{r_m}}{\phi(m)}>\frac{b}{2}\Bigr)
		\\
		&&
		\qquad
		\leq
		P_0^V\Bigl( \limsup_{m\to\infty}\frac{M_{r_m}}{\phi\bigl(r_m\bigr)}>\frac{b}{2}\Bigr)
		\leq
		P_0^V\Bigl( \limsup_{n\to\infty}\frac{M_n}{\phi(n)}>\frac{b}{2}\Bigr).
		\feq
		It follows from \eqref{Mlemma} and \eqref{phi1} that
		\beq
		\sum_{n=1}^\infty P_0^V\Bigl( \frac{M_n}{\phi(n)}>\frac{b}{2}\Bigr)  <\infty.
		\feq
		Thus the Borel-Cantelli lemma and \eqref{lsz} imply
		\beq
		P_0\Bigl( \limsup_{m\to\infty}\frac{\xi^*(T_m)}{\phi(T_{m-1})} > b \Bigr)
		\leq  P_0^V\Bigl( \limsup_{m\to\infty}\frac{\max_{0\leq k\leq m} M_k}{\phi(m-1)}>\frac{b}{2}\Bigr)=0.
		\feq
		By virtue of \eqref{26a}, this completes the proof of the first half of part~(i) of Theorem~\ref{thm1}.
		\\
		$\mbox{}$
		\\
		Suppose now that
		\beqn
		\label{phi}
		\sum_{m=1}^\infty{\frac{1}{[\phi(m)]^{\delta}}} = \infty.
		\feqn
		Recall \eqref{rwspeed} and the fact that $\speed>0$ (the asymptotic speed is positive) when $\delta>2.$
		Observe that in order to prove that
		\beq
		\limsup_{n\to\infty}\frac{\xi^*_n}{\phi(n)} = \infty,\qquad P_0-\as,
		\feq
		it suffices to show that for any constant $b>0,$
		\beqn
		\label{lsw}
		\limsup_{m\to\infty} \frac{\max_{0\leq i <\varrho_m}M_i}{\phi(\lfloor 2\speed^{-1} m \rfloor )} > b, \qquad P_0^V-\as
		\feqn
		Indeed, $P_0\bigl(T_m > 2\speed^{-1}m \io\bigr)=0$ by virtue of \eqref{rwspeed1}. Hence
		\eqref{lsw} implies that for every $b>0,$
		\beq
		&& P_0\Bigl(\limsup_{m\to\infty} \frac{\xi^*(m)}{\phi(m)}>b\Bigr) \geq P_0\Bigl(\limsup_{m\to\infty} \frac{\xi^*(T_m)}{\phi(T_m)}>b\Bigr)
		\\
		&&
		\qquad
		\geq P_0\Bigl(\limsup_{m\to\infty} \frac{\xi^*(T_m)}{\phi(\lfloor 2\speed^{-1} m \rfloor)}>b\Bigr)
		\geq P_0^V\Bigl(\limsup_{m\to\infty} \frac{\max_{0\leq i<\varrho_{m}}M_i}{\phi(\lfloor 2\speed^{-1} m \rfloor)} > b \Bigr)=1,
		\feq
		where the last inequality is justified by Lemma~\ref{triangular-renewal}; the finite residual cycle in the stationary representation is immaterial for the displayed limsup event. Thus the proof of part (i) of Theorem~\ref{thm1} will be completed once we prove \eqref{lsw}.
		\par
		We turn now to the proof of \eqref{lsw}. Recall the notation $\mu=E_0^V[\sigma_0]$ that we have used before.
		By the renewal theorem $P_0^V\bigl(\lim_{m\to\infty}\frac{\varrho_m}{m}=\frac{1}{\mu}\bigr)=1$, and hence $P_0^V\bigl( \varrho_m < \frac{m}{2\mu} \io\bigr)=0.$
		It thus suffices to show that
		\beqn
		\label{lsw4}
		\limsup_{m\to\infty} \frac{\max_{0\leq i \leq \lfloor m\mu^{-1}/2 \rfloor }M_i}{\phi(\lfloor 2\speed^{-1} m \rfloor )}
		\geq  \limsup_{m\to\infty} \frac{M_{\lfloor m\mu^{-1}/2 \rfloor}}{\phi(\lfloor 2\speed^{-1} m \rfloor )} >b, \qquad P_0^V-\as
		\feqn
		To this end, observe that by virtue of \eqref{Mlemma} and \eqref{phi},
		\beq
		\sum_{m=1}^\infty P_0^V\Bigl( M_{\lfloor m\mu^{-1}/2 \rfloor} > b \phi(\lfloor 2\speed^{-1} m \rfloor)\Bigr) = \infty.
		\feq
		The map $m\mapsto \lfloor m\mu^{-1}/2\rfloor$ has uniformly bounded multiplicities.  Hence the divergent series above contains a divergent subseries indexed by distinct values of $n=\lfloor m\mu^{-1}/2\rfloor.$  Along this subseries the corresponding events depend on distinct members of the i.i.d. sequence $(M_n)$ and are therefore independent.  Regular variation of $\phi$ ensures that replacing $\lfloor 2\speed^{-1}m\rfloor$ by a constant multiple of $n$ changes the tail estimates only by multiplicative constants.  The second Borel-Cantelli lemma therefore gives \eqref{lsw4}, and hence \eqref{lsw}. The proof of part~(i) of Theorem~\ref{thm1} is completed.
	\end{proof}
	$\mbox{}$
	\begin{proof} [\textbf{Part (ii)}]
		Let $\phi(n)=n^{1/\delta}(\log n)^{-\alpha}$ for a fixed constant $\alpha>1/\delta.$ In order to prove the claim, it suffices to verify \eqref{claim3}
		(see the next two lines below \eqref{claim3}). Toward this end, observe that according to the law of large numbers for $T_n$
		stated in \eqref{rwspeed1} we have $P_0(T_m > 2\speed^{-1}m\io)=0,$ and hence
		\beq
		\liminf_{m\to\infty}\frac{\xi^*(T_m)}{\phi(2\speed^{-1} (m+1))} \leq \liminf_{m\to\infty}\frac{\xi^*(T_m)}{\phi(T_{m+1})}, \qquad  P_0-\as
		\feq
		Thus it suffices to show that the left-hand side of the above inequality is infinity.
		Let $a_m^-$ be as defined in \eqref{a-m-def} with the only exception that this time we will use an arbitrary constant $\delta'\in(1,2).$
		In view of \eqref{sigma_law} and \eqref{endi}, Chebyshev's inequality implies
		\beqn
		\label{an-ineq5}
		P_0^V(a_m^- >\varrho_{m}) \leq \frac{a_m^-+1}{m^{2/\delta'}} \cdot E_0^V\bigl[ (\sigma_0-\mu)^2\bigr].
		\feqn
		Therefore, a slight modification of \eqref{44a} (namely, formally replacing there the composition of two functions $\phi\circ \psi_\veps$ by
		the ``new" $\phi(n)=n^{1/\delta}(\log n)^{-\alpha}$ and also using \eqref{an-ineq5} instead of \eqref{an-ineq4}) along with the Borel-Cantelli lemma imply that
		$\liminf_{i\to\infty}\frac{\xi^*(T_{m_i})}{\phi(m_{i+1}+1)} =\infty,$ $P_0-\as.$
		This completes the proof of \eqref{claim3} by using an appropriate variation of \eqref{faccord}.
	\end{proof}
	

\begin{thebibliography}{100}
		\bibitem{bsingh1}
		A.-L.~Basdevant and A.~Singh,
		\emph{On the speed of a cookie random walk},
		Probab. Theory Related Fields \textbf{141} (2008), 625--645.
		
		\bibitem{bsingh2}
		A.-L.~Basdevant and A.~Singh,
		\emph{Rate of growth of a transient cookie random walk},
		Electron. J. Probab. \textbf{13} (2008), 811--851.
		
		\bibitem{bwilson}
		I.~Benjamini and D.~B.~Wilson,
		\emph{Excited random walk}, Electron. Comm. Probab. \textbf{8} (2003), 86--92.
		
		\bibitem{rv-book}
		N.~H.~Bingham, C.~M.~Goldie, and J.~L.~Teugels,
		\emph{Regular Variation}. Encyclopedia of Mathematics and its Applications \textbf{27},
		Cambridge University Press, 1989 (paperback edition).
		
		\bibitem{rwbook}
		A.~A.~Borovkov and K.~A.~Borovkov, \emph{Asymptotic Analysis of Random Walks: Heavy-
			tailed Distributions}, Encyclopedia of Mathematics and its Applications \textbf{118}, Cambridge
		University Press, 2008.
		
		\bibitem{nina2}
		A.~Dembo, N.~Gantert, Y.~Peres, and Z.~Shi,
		\emph{Valleys and the maximum local time for random walk in random environment},
		Probab. Theory and Related Fields \textbf{137} (2007), 443--473.
		
		\bibitem{erwnotes}
		D.~Dolgopyat,
		\emph{Random walks in one dimensional environment}, lecture notes.
		Available electronically at \url{http://www2.math.umd.edu/~dolgop/RW1d.pdf}.
		
		\bibitem{nina1}
		N.~Gantert and Z.~Shi,
		\emph{Many visits to a single site by a transient random walk in random environment},
		Stoch. Proc. Appl. \textbf{99} (2002), 159--176.
		
		\bibitem{nina3}
		N.~Gantert, Y.~Peres, and Z.~Shi,
		\emph{The infinite valley for a recurrent random walk in random environment},
		Ann. Inst. H. Poincar\'{e} Probab. Statist. \textbf{46} (2010), 525--536.
		
		\bibitem{key-br}
		E.~S.~Key,
		\emph{Limiting distributions and regeneration times for multitype branching processes with immigration in a random environment},
		Ann. Probab. \textbf{15} (1987), 344--353.
		
		
		\bibitem{kosygina2}
		E.~Kosygina and T.~Mountford,
		\emph{Limit laws of transient excited random walks on integers},
		Ann. Inst. H. Poincar\'{e} Probab. Statist. \textbf{47} (2011), 575--600.
		
		\bibitem{kosygina1}
		E.~Kosygina and M.~Zerner,
		\emph{Positively and negatively excited random walks on integers, with branching processes},
		Electron. J. Prob. \textbf{13} (2008), 1952--1979.
		
		\bibitem{kosygina3}
		E.~Kosygina and M.~P.~W.~Zerner,
		\emph{Excited random walks: results, methods, open problems},
		Bulletin of the Institute of Mathematics. Academia Sinica (New Series) 8 (2013), 105--157.
		
		\bibitem{revesz1}
		P. Revesz,
		\emph{Random Walk in Random and Non-Random Environments},
		World Scientific Publishing Company, 2005.
		
		\bibitem{multi}
		A.~Roitershtein,
		\emph{A note on multitype branching processes with immigration in a random environment},
		Ann. Probab. \textbf{35} (2007), 1573--1592.
		
		\bibitem{martin1}
		M. Zerner,
		\emph{Multi-excited random walks on integers},
		Probab. Theory Related Fields \textbf{133} (2005), 98--122.
		
	\end{thebibliography}
\end{document}